\documentclass[oneside]{amsart}
\usepackage[colorlinks=true,allcolors=blue!80!black]{hyperref}
\usepackage{tikz, tikz-cd}
\usetikzlibrary{matrix, snakes, patterns, shadows.blur, backgrounds}

\usepackage{amssymb, amsthm}
\usepackage{enumerate, amsfonts, amsxtra,amsmath,latexsym,epsfig, color}
\usepackage{mathrsfs, MnSymbol}
\usepackage{geometry}                		
\usepackage{graphicx}
\usepackage{subcaption}
\usepackage{autobreak}
\usepackage{color}
\usepackage{amsmath,amscd}	
\usepackage{mathtools}
\usepackage{stmaryrd}
\usepackage{animate}
\usepackage{tikz-cd}
\usepackage{soul}
\usepackage{faktor}

\usepackage{xypic}

\input xy
\xyoption{all}

\newtheorem {theorem}{Theorem}[section]
\newtheorem {lemma} [theorem] {Lemma}
\newtheorem {proposition} [theorem] {Proposition}

\newtheorem {property} [theorem] {Property}

\theoremstyle{definition}
\newtheorem{definition}[theorem]{Definition}
\newtheorem{remark}[theorem]{Remark}

\def\C {\mathbb C}
\def\R {\mathbb R}

\def\Hy{\mathbf{H}}
\def\HC{\mathbf{H}_\C}

\def\mc {\mathcal}

\def\Z {\mathbb{Z}}

\def\cP{\mathcal{P}}

\DeclareMathOperator{\PU}{PU}

\begin{document}
	\title[complex hyperbolic lattices and (relative) strict hyperbolization]{A note on complex hyperbolic lattices and strict hyperbolization}
	

	\author{Kejia Zhu}
	\address{Department of Mathematics\\ University of California at Riverside}
	\email{kzhumath@gmail.com}
	\urladdr{https://sites.google.com/view/kejiazhu}

	\date{}
	\maketitle
	\begin{abstract}
		We study the connection between the fundamental groups of complex hyperbolic manifolds and those of spaces arising from the (relative) strict hyperbolization process due to Charney--Davis and Davis--Januszkiewicz--Weinberger. Viewing a non-uniform lattice \(\Gamma\) in \(\text{PU}(n,1)\) as a relatively hyperbolic group with respect to its cusp subgroups in the usual way, we show that when $n\geq 2$, \(\Gamma\) is not isomorphic to
		any relatively hyperbolic group arising from the relative strict hyperbolization process, via work of Lafont--Ruffoni. We also prove that a uniform lattice in $\text{PU}(n,1)$ is not the fundamental group of a Charney-Davis strict hyperbolization when $n\geq 2$, assuming the initial complex satisfies some mild conditions.
	\end{abstract}

	\section{Introduction}

	The strict hyperbolization introduced by Charney--Davis \cite{charney1995strict} gives rich examples of closed aspherical manifolds with hyperbolic fundamental groups. More generally, their hyperbolization procedure turns a finite simplicial complex $K$ into a strictly negatively curved piecewise hyperbolic finite cell complex $\mathcal H(K)$. There is a relative version, due to Charney--Davis \cite{charney1995strict} and Davis--Januszkiewicz--Weinberger \cite{davis2001relative}, known as the relative strict hyperbolization \cite{charney1995strict} that constructs compact aspherical manifolds with prescribed boundary as follows. If $K$ is a finite simplicial complex with subcomplex $L$, we may associate to \((K,L)\) a topological space \(\mathcal{R}=\mathcal{R}(K,L)\) with boundary. It is known that the boundary of \(\mathcal{R}\) may be identified with \(L\) up to subdivision, and with this identification, the components of \(L\) \(\pi_1\)-inject \(\mathcal{R}\). Let $\cP_\mathcal{R}$ be a set of representatives of the conjugacy classes of the fundamental groups of the components of $L$, viewed as subgroups of $\pi_1(\mathcal R)$. Then the group $\pi_1(\mathcal R)$ is known to be relatively hyperbolic with respect to $\mathcal P_\mathcal{R}$. See \cite{groves2023relative} and \cite{belegradek2007aspherical} for details. It is generally unknown whether the spaces arising from the relative strict hyperbolization process are complex hyperbolic. In this note, we show the following result:
	
	\begin{theorem}\label{main}
		The process of strict hyperbolization of a compact homogeneous $n$-dimensional simplicial complex without boundary and the process of relative strict hyperbolization of a finite simplicial complex with subcomplex cannot produce complexes with fundamental groups isomorphic to complex hyperbolic lattices of (complex) dimension $n\ge2$, in particular, the process of (relative) strict hyperbolization above cannot produce complex hyperbolic manifolds of (complex) dimension at least $2$.
	\end{theorem}
	
	\begin{remark}
		Recall, by definition, a complex hyperbolic lattice is a discrete subgroup of $\text{PU}(n,1)$ with finite quotient measure, and a complex hyperbolic manifold is the quotient space of the complex hyperbolic $n$-space $\Hy_{\mathbb C}^n$ by a torsion-free lattice in the $\text{PU}(n,1)$, so the proof of Theorem\ref{main} follows quickly from the Theorem \ref{maintheoremnew} and Theorem \ref{hyperboliclr}.
	\end{remark}

	
	The proof for the relative case (Theorem \ref{maintheoremnew}) depends on the work of Lafont--Ruffoni. Lafont--Ruffoni proved the following Theorem \ref{themA} regarding the properties that the relative hyperbolic group  \((\pi_1(\mathcal{R}), \mathcal{P}_\mathcal{R})\) would satisfy: Given a relatively hyperbolic group $(G,\mathcal{P})$ acting isometrically on a CAT(0) cube complex $X$, we define a property below that $(G,\mathcal{P})$ arises as \((\pi_1(\mathcal{R}), \mathcal{P}_\mathcal{R})\) would satisfy:
	\begin{property}\label{property}
		\noindent\\
		(1) $X/G$ is compact.\\
		(2) Each $p\in \mathcal P$ acts elliptically on $X$.\\
		(3) For each cube $C$ of $X$, $\text{Stab}_G(C)$ is either maximal parabolic, or else is full relatively quasi-convex, hyperbolic, and virtually compact special.
	\end{property}
	
	\begin{theorem}[\cite{groves2023relative}-Theorem A]\label{themA}
		The relative hyperbolic group \((\pi_1(\mathcal{R}), \mathcal{P}_\mathcal{R})\) admits an action on a CAT(0) cube complex \(X\) isometrically satisfying the Property \ref{property}
	\end{theorem}

	\begin{remark}
		Recall that if $(\Gamma,\cP)$ is a relatively hyperbolic group pair and $\Gamma_0\leq \Gamma$ has finite index then $(\Gamma_0,\cP_0)$ is also a relatively hyperbolic group pair, where $\cP_0$ is the set of representatives of the conjugacy classes of\begin{equation}\label{eqn1}\{P^g\cap\Gamma_0|g\in \Gamma, P\in \mathcal P\}.\end{equation} Suppose $(\Gamma,\cP)$ has Property \ref{property}, then it follows immediately that $(\Gamma_0,\cP_0)$ has Property \ref{property}. We will thus freely pass to finite index subgroups in the proof.
	\end{remark}

	Note Property \ref{property} closely resembles those of a relative geometric action, introduced by Einstein and Groves (\cite{einstein2020relative}). The difference lies in $(3)$, where relatively geometric actions require a cell stabilizer to be either finite or conjugate to a finite index subgroup of a peripheral subgroup. In \cite{bregmanrelatively}, Bregman, Groves, and the author proved that \((\Gamma,\mathcal{P})\) does not admit a relative geometric action on a CAT(0) cube complex when \(\Gamma\) is a non-uniform lattice in \(\text{PU}(n,1)\), and \(\cP\) is a set of conjugacy class representatives for the cusp subgroups of \(\Gamma\). Recall we are interested in the chance that relatively strict hyperbolized objects from Lafont--Ruffoni can be complex hyperbolic, it is thus natural to consider whether such pairs can arise from the relative strict hyperbolization considered by Lafont--Ruffoni. The following theorem gives a negative answer. 
	
	\begin{theorem}\label{maintheoremnew}
		When $n\ge2$, a non-uniform lattice of $\PU(n,1)$ as relatively hyperbolic groups (relative to the cusp subgroups), does not satisfy Property \ref{property}, in particular, it does not arise as $(\pi_1(\mathcal R(K, L)),\mathcal P_\mathcal{R})$.
	\end{theorem}

	Now given a uniform lattice $\Gamma$ in \(\text{PU}(n,1)\), by Selberg’s lemma, we can always pass it to a torsion-free finite index subgroup. Then it acts freely properly on the complex hyperbolic $n$-space $\HC^n$ and the quotient space is a compact K\"ahler manifold with negative curvature. Recall a K\"ahler group is the fundamental group of a compact K\"ahler manifold. In particular, $\Gamma$ is a K\"ahler group. Also recall a cubulated group is a group acting geometrically on a CAT(0) cube complex. Delzant--Gromov in \cite{delzant2005cuts}[Corollary, p.52] showed that any K\"ahler group that admits a full system of convex cuts is commensurable to a surface group, where a convex cut is a codimension one quasi-convex subgroup. Recall a virtually special hyperbolic group is cubulated hyperbolic (see \cite{wise2021structure}-Lemma 7.14), in particular, it admits a full system of convex cuts,
	we apply the result of Delzant–Gromov to the following result of Lafont and Ruffoni \cite[Theorem 1.1 and 1.2]{lafont2022special} to get the non-relative version of the Theorem \ref{themA} (see
	Theorem \ref{hyperboliclr}):
	
	Let $n > 0$ and $K$ an \(n\)-dimensional simplicial complex which is compact, homogeneous, and without boundary. If $\mathcal H(K)$ is the Charney--Davis strict hyperbolization of $K$, then $\pi_1(\mathcal H(K))$ is Gromov hyperbolic and virtually compact special.
	
	\begin{theorem}\label{hyperboliclr}
		If \(K\) is a compact homogeneous $n$-dimensional simplicial complex without boundary, then \(\pi_1(\mathcal{H}(K))\) is not commensurable to a K\"ahler group if its cohomology dimension is strictly greater than $2$. In particular,  \(\pi_1(\mathcal{H}(K))\) is not isomorphic to a uniform lattice in $\PU(n,1)$ for $n\ge2$.
	\end{theorem}
	\begin{proof}
		By Lafont and Ruffoni \cite[Theorem 1.1 and 1.2]{lafont2022special}, \(\pi_1(\mathcal{H}(K))\) is 
		Gromov hyperbolic and virtually special. Suppose \(\pi_1(\mathcal{H}(K))\) is commensurable to a K\"ahler group, then it has a finite index subgroup, say $G$, which is K\"ahler, Gromov hyperbolic and virtually special (recall a finite index subgroup of a K\"ahler is still K\"ahler). Now by \cite{delzant2005cuts}[Corollary, p.52], $G$, thus \(\pi_1(\mathcal{H}(K))\), is commensurable to a surface group, whose cohomological dimension is $2$, so the cohomological dimension of \(\pi_1(\mathcal{H}(K))\) can not be strictly greater than $2$. In particular, \(\pi_1(\mathcal{H}(K))\) is not isomorphic to a uniform lattice $\Gamma<\PU(n,1)$ for $n\ge2$ because otherwise by Selberg's lemma, $\Gamma$ has a finite index torsion-free subgroup, which is K\"ahler and with cohomological dimension $2n\ge 4$, therefore \(\pi_1(\mathcal{H}(K))\) is commensurable to a K\"ahler group of cohomological dimension $2n\ge 4$, contradiction.
	\end{proof}
	
	\begin{remark}
		Belegradek in \cite[Theorem 10.3]{belegradek2007aspherical} showed that if $K$ is a finite simplicial complex, then $\mathcal H(K)$ is not homotopy equivalent to a compact  K\"ahler manifold of complex dimension $n\ge2$. The theorem \ref{hyperboliclr} gives a stronger result in the sense of fundamental groups by assuming $K$ satisfies some mild conditions. If the relative strict hyperbolization $\mathcal{R}(K,L)$ is an $n$-manifold, Belegradek in \cite[Theorem 10.5]{belegradek2007aspherical} proved that the interior of $\mathcal{R}(K,L)$ is not homeomorphic to an open set of a compact K\"ahler manifold of complex dimension $m\ge2$. The theorem \ref{maintheoremnew} directly deals with the fundamental group of the whole $\mathcal{R}(K,L)$ without assuming $\mathcal{R}(K,L)$ to be a manifold.
		
	\end{remark}

	\noindent
	\textbf{Outline:} The remainder of this note is devoted to the proof of \ref{maintheoremnew}. In Section \ref{dehnfilling}, we first introduce the Dehn filling and related definitions. Then we apply some techniques from Groves and Manning to prove a key technical result, i.e. the Theorem \ref{theorem:Dehn}. In Section \ref{secCompact}, we first review the Borel--Serre and Toroidal compactifications and construct three compact Eilenberg–MacLane spaces by filling cusps, whose fundamental groups are sitting in a factorization. In Section \ref{ending}, we prove Theorem \ref{maintheoremnew} by applying the result of Delzant--Gromov to the factorization above to deduce a contraction if we assume $(\Gamma,\cP)$ satisfies Property \ref{property}.\\

	\noindent\textbf{Acknowledgments:} The author would like to thank his advisor, Daniel Groves, for answering his questions. He also would like to thank Corey Bregman, Matthew Durham, Shaver Phagan, and Lorenzo Ruffoni for their helpful comments. In particular, the author would like to thank Corey Bregman for permitting the author to employ some methods and techniques from our joint work in an unpublished version of \cite{bregmanrelatively} which are not used in the final version with Daniel Groves.

	\section{Relatively Hyperbolic Dehn Filling}\label{dehnfilling}
	Dehn filling is a well-known procedure for changing hyperbolic structures in $3$-manifold topology (\cite[Section 5.8]{thurston2022geometry}). Mimicking the effect of the classical Dehn filling of a manifold on its fundamental group, Osin \cite{osin2007peripheral} and Groves and Manning \cite{GrovesManningDehn} developed a generalized notion of Dehn filling, which is entirely group-theoretic and a-priori devoid of any geometric content. We now recall some basic definitions associated with this generalized notion of a Dehn filling of a group pair $(G,\mathcal{P})$, which are frequently used in the techniques developed by Groves and Manning \cite{GrovesManningDehn}:\\
	
	Given a group pair $(G,\mathcal P)$, where $\mathcal P = \{P_1, ..., P_m\}$ and we choose a normal subgroup $N_i$ for each peripheral group $P_i$. Define $\mathcal{N}:=\{N_i\}$. Let $K$ be the normal closure generated by the union $\bigcup N_i$.
	\begin{definition}[Dehn Filling, \cite{GrovesManningDehn}]
		The \emph{Dehn filling} of $(G,\mathcal P)$ with respect to $\mathcal{N}$ is the pair $(\overline{G},\overline{\mathcal P})$ from the quotient $G\to G/K=:\overline{G}$, where $\bar{\mathcal P}$ is the set of images of $\mathcal P$ under this quotient. The elements of $\mathcal N$ are called the \emph{filling kernels}. When we want to specify the filling kernels we will write $G(N_1,\ldots, N_m)$ for the quotient $\overline{G}$. The filling is \emph{peripherally finite} if each normal subgroup $N_i$ has finite index in $P_i$. Given a subgroup $Q<G$, we call $G(N_1,\ldots,N_m)$ a \emph{$Q$-filling} if for all $g \in G$, and $P_i\in \cP$,  $|Q \cap P_i^g|=\infty$ implies $N_i^g\subseteq Q$. If $\mc{Q}=\{Q_1,\ldots, Q_l\}$ is a family of subgroups, then $G(N_1,\ldots, N_m)$ is a \emph{$\mc{Q}$–filling} if it is a $Q$–filling for every $Q\in \mc{Q}$. We say that a property $\mathcal X$ holds for all \emph{sufficiently long} Dehn fillings of $(G,\mathcal P)$ if there is a finite subset $X\subset G\setminus\{1\}$ so that whenever $N_i \cap B =\emptyset$ for all $i$, the corresponding Dehn filling $G(N_1,..., N_n)$ has property $\mathcal X$.
	\end{definition}
	
	\begin{definition}
		A subgroup $G_1$ of $G$ is \emph{full} if whenever $P$ is a parabolic subgroup of $G$ such that $|G_1\cap  P|$ is infinite, we have $G_1\cap P$ is a finite index subgroup of $P$. 
	\end{definition}

	Let $\cP_0=\{P_{0,1},\ldots,P_{0,r}\}$ be the collection of peripheral subgroups of \(\overline{\Gamma}_0\) induced by $\cP$ as in Equation\ref{eqn1}. The following theorem is regarding Property \ref{property}. The proof of part (1), in particular, the filling, is motivated by the proof of Theorem A4 in the appendix of \cite{groves2023relative} by Groves and Manning.

	\begin{theorem}\label{theorem:Dehn}Let $(\Gamma,\cP)$ be a relatively hyperbolic group with peripheral subgroups residually finite. If $(\Gamma,\mathcal P)$ admits an action on a CAT(0) cube complex $X$ satisfying Property \ref{property}. Then there is a finite index subgroup $\Gamma_0\leq \Gamma$ and a Dehn filling $\psi:\Gamma_0\rightarrow\overline{\Gamma}_0$ satisfying: 
		\begin{enumerate}
			\item $\overline{\Gamma}_0$ is hyperbolic and virtually special 
			\item $K:=\ker\psi$ is the normal closure of the subgroup generated by the peripheral subgroups in $\mathcal{P}_0$, and \(K\cong \underset{g\in \overline \Gamma_0}{\text{\Large$\ast$}}(P_{0,1}*\cdots*P_{0,r})\), where the inner free product is over $\overline \Gamma_0$-orbits of elements of $\cP_0$.
		\end{enumerate}
	\end{theorem}

	\begin{proof} For (1), the strategy is to construct a sufficiently long peripherally finite $\mathcal Q$-filling (see the second paragraph) and then apply the Corollary 6.6 of \cite{groves2018hyperbolic}. Suppose $(\Gamma,\cP)$ satisfies the Property \ref{property} on a CAT(0) cube complex $X$. Since $X/\Gamma$ is compact, there are only finitely many $\Gamma$-orbits of cubes in $X$. By assumption, the stabilizer of each cube is full relatively quasi-convex (recall maximal parabolic subgroups are naturally full relatively quasi-convex), and conversely, each element of $P_j\in \cP$ stabilizes some cube. Therefore, the hypotheses of Corollary 6.6 of \cite{groves2018hyperbolic} is satisfied by the action of $\Gamma$ on $X$.

		Let $T_1,\ldots, T_k$ be representatives of the $\Gamma$-orbits of cubes in $X$, and let $\mathcal{Q}=\{Q_1,\ldots, Q_k\}$, where $Q_i$ is the finite index subgroup of the stabilizer of $T_i$ fixing $T_i$ pointwise. Now we construct a sufficiently long $\mathcal Q$-filling: for each $P_j\in\mathcal{P}$ there are finitely many infinite intersections $P_j\cap Q_i^g$ for $Q_i\in\mathcal{Q}, g\in G$; each such $P_j\cap Q_i^g$ is finite index in $\mathcal{P}$ (recall each $Q_i$ as a cell stabilizer is full relatively quasi-convex). Let $N$ be the normal core of the intersection of these $P_j\cap Q_i^g$, where the normal core here is the largest normal subgroup of $\Gamma$ that is contained in the intersection. Given any collection of finite index $\{N_j'\unlhd P_j\}$, the collection $\{N_j:=N_j'\cap N\}$ gives a peripherally finite $\mathcal{Q}'$–filling. Since each $P_j$ is residually finite, we are able to choose the $N_j'$ to avoid any finite subset of $G\setminus \{1\}$, so the $\mathcal{Q}'$–filling is also sufficiently long.

		Now we apply the Corollary 6.6 of \cite{groves2018hyperbolic} to this sufficiently long $\mathcal Q'$-filling, which asserts that the quotient $\overline X=K\backslash X$ is a CAT(0) cube complex. Since the filling is peripherally finite, the peripheral subgroups of $(\overline{\Gamma},\overline{\cP})$ are finite, thus $\overline{\Gamma}$ is hyperbolic. 
		
		Observe that $\bar{X}/\overline{\Gamma}=(X/K)/(\Gamma/K)=X/\Gamma$ is compact, so $\overline{\Gamma}$ acts cocompactly on $\bar{X}=K\backslash X$. Also, note the stabilizers are either finite or quasi-convex and virtually special. So the Theorem D of \cite{groves2018hyperbolic} now asserts that $\overline{\Gamma}$ is virtually special, thus $\overline{\Gamma}$ is residually finite and virtually torsion-free. Let $\overline \Gamma_0$ be a finite index torsion-free normal subgroup of $\overline \Gamma$, in particular, $\overline{\Gamma}_0$ is hyperbolic and virtually special. This proves (1).
		
		Define $\Gamma_0$ to be the preimage of $\overline{\Gamma}_0$ in the filling $K\to \Gamma\to \overline\Gamma=\Gamma(N_1,...,N_n)$, in particular $\Gamma_0$ is a finite index subgroup of $\Gamma$ that is hyperbolic relative to $\mathcal{P}_0$ defined as in Equation\ref{eqn1}. Also note that $K$ is contained in $\Gamma_0$.

		For (2), we first claim that each peripheral subgroup group $P_{0,i}$ of $\Gamma_0$ is conjugate to one of the filling kernels $N_j$: recall by Equation\ref{eqn1}, each $P_{0,i}=\Gamma_0\cap P_j^g$ for some $P_j\in \cP$ and $g\in \Gamma$. Recall $N_j^g< K\leq \Gamma_0$ and $N_j^g$ is a finite index subgroup of $P_j^g$, so $N_j^g$ is a finite index subgroup of $\Gamma_0\cap P_j^g=P_{0,i}$.  On the other hand, by Theorem 1.1 of \cite{osin2007peripheral}, we may assume that $P_j/N_j\to \overline{\Gamma}$ is injective, thus $P_j^g/N_j^g\to\overline{\Gamma}$ is also injective. Recall $\overline\Gamma_0<\overline\Gamma$ is torsion-free and $P_j^g/N_j^g$ is finite, so the injection implies that $ P_{0,i}=\Gamma_0\cap P_j^g\subset N_j^g$, it follows that $N_j^g= P_{0,i}$, as desiorange. In particular, $\llangle \bigcup_i P_{0,i}\rrangle=\llangle \bigcup_j N_{j}^g\rrangle$. Therefore, $$K=\llangle \bigcup_i N_{i}\rrangle=\llangle \bigcup_i N_{i}^g\rrangle=\llangle \bigcup_i P_{0,i}\rrangle.$$ Now by the Theorem 4.8 of \cite{GrovesManningSisto}, for a sufficiently long Dehn filling, we have that \[K\cong \underset{g\in \overline{\Gamma}}{\text{\Large$\ast$}}(N_1*\cdots*N_m)\] Under the conjugation action of $\overline{\Gamma}_0$, we can rewrite this as 
		\begin{align*}K&\cong \underset{g\in \overline{\Gamma}_0}{\text{\Large$\ast$}}(*_{\bar{t}\in \overline{\Gamma}/\overline{\Gamma}_0} (N_1*\cdots*N_m))\\
			&\cong \underset{g\in \overline{\Gamma}_0}{\text{\Large$\ast$}}(*_{t\in \Gamma/\Gamma_0} (N_1*\cdots*N_m))\\
			&\cong \underset{g\in \overline \Gamma_0}{\text{\Large$\ast$}}(P_{0,1}*\cdots*P_{0,r}),
		\end{align*}
		where the last isomorphism follows from that each $P_{0,i}$ is a conjugate of some $N_j$ and when $t$ ranges over elements of $\Gamma/\Gamma_0$, the conjugates $N_j$ in the inner free product above range over peripheral subgroups $P_{0,i}$.
	\end{proof}

	\section{Compactification of $M$}\label{secCompact}
	It is well-known that a non-uniform lattice $\Gamma$ in $\PU(n,1)$ is virtually torsion-free. So we can assume $\Gamma$ is torsion-free by passing to a finite-index subgroup if necessary. Therefore, the quotient space \(M:=\Gamma\backslash\Hy^n_\C\) is a non-compact hyperbolic complex manifold of finite volume with finitely many cusps.

	\subsection{The structure of cusps}\label{sec:Cusp Structure}
	\quad\\
	Here we give a concise description of the cusps in $M$, for more details see \cite[Section 4.1]{bregmanrelatively}. Let $\{\mathcal{C}_1,\ldots, \mathcal{C}_r\}$ be the cusps of $M$. By the Margulis lemma, each cusp $\mathcal{C}_i$ of $M$ is homeomorphic to $\Sigma_i \times [0,\infty)$ where $\Sigma_i$ is an $(2n-1)$-dimensional manifold as the quotient of $\mc{H}_{2n-1}(\R)$, the $(2n-1)$-dimensional Heisenberg group, by the peripheral subgroup $P_i$ (see \cite[\S 10.3]{BallmannSchroederGromov} for details). Then the corresponding cusp subgroup is isomorphic to $\pi_1(\Sigma_i)$. Each cusp corresponds to a conjugacy class of subgroups fixing a point on the boundary at infinity of \(\Hy^n_\C\). Farb in \cite{Farb-RH} proved that $\Gamma$ is hyperbolic relative to the collection of its cusp subgroups, which we denote by $\cP$. 
	
	There is a fiber sequence
	$$S^1\to \Sigma_i\to \mathcal O_i,$$ where $\mathcal O_i$ is a Euclidean orbifold finitely coveorange by $2n-2$-dimensional torus (see \cite[Section 4.1]{bregmanrelatively} for a comprehensive explanation). Note in general $\mathcal O_i$ can be singular, thus this is not generally a locally trivial fibration. But in our case, since each $P_i$ is
	torsion-free, $\Sigma_i$ is smooth and we freely pass to the torus cover to obtain an actual fiber bundle 
	\begin{equation}\label{circleseq}S^1\to \Sigma_i\to T_i,
	\end{equation}

	In the following three subsections, we will construct three different compact Eilenberg–MacLane spaces by compactifying $M$, which will be used in the Theorem \ref{maintheoremnew}: 
	
	\subsection{Borel-Serre compactification}
	\quad\\
	Borel and Serre in \cite{borel1973corners} constructed a compactification $M_0$ of $M$, which works in general for lattices in noncompact symmetric spaces. In our case, i.e. the rank 1 case and $\Gamma$ is torsion-free, the $M_0$ is a smooth manifold with boundary.   
	
	If we lift $\{\mathcal{C}_i\}$ to the universal cover $\Hy_{\mathbb C}^n$, the pre-image of $\sqcup_i\mathcal{C}_i$ is a  union of pairwise disjoint horoballs centeorange on the parabolic fixed points in $\partial_\infty\Hy_{\mathbb C}^n$, which is $\Gamma$-equivariant. The $\Gamma$-stabilizer of each parabolic fixed point conjugate to some $P_i$ that acts freely, geometrically on the corresponding horosphere with quotient $\Sigma_i$. The quotient of the complement of the interiors of these horoballs is therefore a compact aspherical manifold $M_0$ with boundary $\partial M_0=\sqcup_i\Sigma_i$. In particular, $M_0$ is an Eilenberg–MacLane space $K(\Gamma,1)$, thus $H_i(M_0)\cong H_i(\Gamma)$. The pair $(M_0,\partial M_0)$ is called the \emph{Borel--Serre compactification of $M$}. Each peripheral subgroup $P_i$ is associated with a unique boundary component $\Sigma_i$ and vice versa. Thus, $M_0$ is obtained from $M$ by removing a neighborhood of the cusps, in particular, $M$ and $M_0$ are homotopy equivalent and thus $\pi_1(M_0)=\pi_1(M)=\Gamma$. 
	
	\subsection{Toroidal compactification}
	\quad\\
	There is another natural compactification of $M$ which fills in the cusps with the Euclidean orbifolds called \emph{Toroidal compactification}. The Toroidal compactification of $M$, is given by filling each fiber in the circle bundle $S^1\to \Sigma_i\to T_i$ to a disk $\mathbb D^2$, resulting a $\mathbb D^2$-bundle over $T_i$: $$\mathbb D^2\to B_i\to T_i,$$ where $B_i$ is the total space, which admits a deformation retraction to $T_i$. The resulting space is a K\"ahler manifold $\mathcal{T}(M)$ with boundary $T:=\sqcup T_i$. Moreover, $\partial B_i\cong \Sigma_i$.
	
	Hummel--Schroeder (see \cite{HummelSchroeder}) showed that we can always pass to a finite cover of $M$ so that its Toroidal compactification is an aspherical compact K\"ahler manifold with a smooth boundary $T$. In the language of Dehn-filling, Toroidal compactification (after passing to a finite cover of $M$ if necessary) is given by collapsing the normal subgroup of each cusp subgroup corresponding to the $S^1$ in each circle bundle $S^1\to \Sigma_i\to T_i$. For a more comprehensive description of Toroidal compactification, see \cite[Section 4.2]{bregmanrelatively}.

	For the following, we will assume that $\mathcal{T}(M)$ is an aspherical compact K\"ahler manifold with a smooth boundary $T$. In particular, $\mathcal{T}(M)$ is an Eilenberg–MacLane space $K(\mathcal{T}(M),1)$, thus $H_i(\mathcal{T}(M))\cong H_i(\pi_1(\mathcal{T}(M))$. Observe that $M_0\setminus \partial M_0$ is homeomrphic to $M$ and $M$ is homeomorphic to $\mathcal{T}(M)\setminus T$, which induces a map of relative pairs $f: (M_0, \partial M_0)\to (\mathcal{T}(M),T) $ that is a homeomorphism on the interior of $M_0$ and sends each boundary component $\Sigma_i\subset \partial M=\sqcup \Sigma_i$ to the corresponding boundary component $T_i\subset T=\sqcup_i T_i$ by the fibering in Equation\ref{circleseq}.

	\begin{proposition}\label{HomologyToroidal} The map $f\colon (M_0, \partial M_0)\rightarrow (\mathcal{T}(M),T) $ induces an isomorphism \[f_*\colon H_{2n}(M_0,\partial M_0)\rightarrow H_{2n}(\mathcal{T}(M),T)\]
	\end{proposition}
	\begin{proof}
		Since $\partial B_i\cong \Sigma_i$, removing the interior of $\sqcup_i B_i$ yields a manifold homeomorphic to $M_0$. By Excision theorem, we thus have $H_*(M_0,\partial M_0)\cong H_*(\mc{T}(M),\sqcup_i B_i)$. On the other hand, each $B_i$ is homotopy equivalent to $T_i$, so $H_*(M_0,\partial M_0)\cong H_*(\mc{T}(M),T)$. Restricted to $M_0\cong \mathcal{T}(M)\setminus \text{Int}(\sqcup_i B_i)$, this deformation retraction gives exactly the map $f$ defined above.
	\end{proof}


	\subsection{Collapsing the cusps}
	\quad\\
	Let $\Gamma_0\leq \Gamma$ be the finite index subgroup whose existence is guaranteed by Theorem \ref{theorem:Dehn}. Replacing $\Gamma$ with $\Gamma_0$, we may assume that there exists a quotient homomorphism $\varphi\colon \Gamma\rightarrow \overline{\Gamma}$ such that 
	\begin{itemize}
		
		\item $\overline{\Gamma}$ is infinite, torsion-free hyperbolic, and virtually special, in particular, $\overline{\Gamma}$ cubulated hyperbolic. 
		\item Let $\cP=\{P_1,\ldots, P_r\}$ the peripheral subgroups. Then $K=\ker\psi$ is the normal closure of the subgroup of $\Gamma$ generated by elements of $\cP$ and moreover $K$ can be written as a free product
		\begin{equation}\label{eqn:freeproduct}K\cong \underset{g\in \overline{\Gamma}}{\text{\Large$\ast$}}(P_{1}*\cdots*P_{r})\end{equation}
	\end{itemize}\quad\\

	Write $\partial M_0=\Sigma_1\sqcup\cdots \sqcup \Sigma_r$, where $\pi_1(\Sigma_i)=P_i$. Since $\Sigma_i$ is a quotient of the Heisenberg group $\mc{H}_{2n-1}(\R)$ by a torsion-free, cocompact subgroup, each $\Sigma_i$ is aspherical.
	
	By construction, $\overline{\Gamma}$ is obtained from $\Gamma$ by killing the peripheral subgroups $P_i$. Now define 
	$$\mc{C}(M):=M_0\sslash\partial M_0$$
	where the notation ``$\sslash$" means that we collapse each component of $\partial M_0$ to a point resepctively. Let $p_i\in \mc{C}(M)$ be the image of $\Sigma_i$ and set $S=\{p_1,\ldots, p_r\}$. So the space obtained from $M_0$ by gluing on the cone $C(\Sigma_i)$ for each component $\Sigma_i\subset \partial M_0$ is homeomorphic to $\mathcal{C}(M)$.

	\begin{figure}[h]
		\centering
		\begin{subfigure}{0.4\textwidth}
			\centering
			\begin{tikzpicture}
				\draw[thick, orange](-.2,3.3) arc(-180:180:.345cm and .15cm);
				\draw[orange](-.02
				,3.3) to[out=-20,in=200](.3,3.3);
				\draw[orange](0,3.29) to[out=20,in=160](.25,3.29);
				
				\draw[thick,orange](1.9,3.3) arc(-180:180:.325cm and .12cm);
				\draw[orange](2.05
				,3.3) to[out=-20,in=200](2.42,3.3);
				\draw[orange](2.1,3.29) to[out=20,in=160](2.38,3.29);
				
				\draw[thick](-0.1,4) to[out=270,in=90]
				(-1,1) to[out=270,in=180](1,0) to[out=0,in=270](3,1) to[out=90,in=270](2.5,4);
				\draw[thick](2,4) to[out=270,in=0](1,2.5) to[out=180,in=270](.5,4);
				\draw[thick](.2,1.5) to[out=-50,in=240](1.8,1.5);
				\draw[thick](.3,1.4) to[out=30,in=140](1.7,1.4);
				
				\draw[thick,blue](-0.1,4) arc(-180:180:.3cm and .15cm);
				\draw[blue](.05
				,4) to[out=-20,in=200](.35,4);
				\draw(0.08,3.99) to[out=20,in=160](.33,3.99);
				
				\draw[thick,blue](2,4) arc(-180:180:.25cm and .12cm);
				\draw[blue](2.1
				,4) to[out=-20,in=200](2.4,4);
				\draw[blue](2.13,3.99) to[out=20,in=160](2.38,3.99);
				
				\node at ((0.2,4.35){$\scriptstyle \Sigma_1$ };
				\node at ((2.25,4.325){$\scriptstyle \Sigma_2$ };
				
			\end{tikzpicture}
			
			\caption{Borel--Serre $M_0$, $\partial M_0=\Sigma_1\bigcup \Sigma_2$}
			\label{fig:BorelSerre}
		\end{subfigure}
		\hfill
		\begin{subfigure}{0.4\textwidth}
			\centering
			\begin{tikzpicture}
				\draw[thick, orange](-.2,3.3) arc(-180:180:.345cm and .15cm);
				\draw[orange](-.02
				,3.3) to[out=-20,in=200](.3,3.3);
				\draw[orange](0,3.29) to[out=20,in=160](.25,3.29);
				
				\draw[thick,orange](1.9,3.3) arc(-180:180:.325cm and .12cm);
				\draw[orange](2.05
				,3.3) to[out=-20,in=200](2.42,3.3);
				\draw[orange](2.1,3.29) to[out=20,in=160](2.38,3.29);
				
				\draw[thick](-0.1,4) to[out=270,in=90]
				(-1,1) to[out=270,in=180](1,0) to[out=0,in=270](3,1) to[out=90,in=270](2.5,4);
				\draw[thick](2,4) to[out=270,in=0](1,2.5) to[out=180,in=270](.5,4);
				\draw[thick](.2,1.5) to[out=-50,in=240](1.8,1.5);
				\draw[thick](.3,1.4) to[out=30,in=140](1.7,1.4);
				
				\draw[very thick, blue](-0.1,4) to[out=-20,in=200](.5,4);
				\filldraw[blue] (-0.1,4) circle (1pt);
				\filldraw[blue] (.5,4) circle (1pt);
				

				\draw[very thick, blue](2,4) to[out=-20,in=200](2.5,4);
				\filldraw[blue] (2,4) circle (1pt);
				\filldraw[blue] (2.5,4) circle (1pt);
				
				\node at ((0.2,4.2){$\scriptstyle T_1$ };
				\node at ((2.25,4.2){$\scriptstyle T_2$ };

			\end{tikzpicture}
			\caption{Toroidal $\mathcal{T}(M)$, $T=T_1\bigcup T_2$}
			\label{fig:Toroidal}
		\end{subfigure}
		\hfill
		\begin{subfigure}{0.5\textwidth}
			\centering
			\begin{tikzpicture}

				\draw[thick, orange](-.2,3.3) arc(-180:180:.345cm and .15cm);
				\draw[orange](-.02
				,3.3) to[out=-20,in=200](.3,3.3);
				\draw[orange](0,3.29) to[out=20,in=160](.25,3.29);
				
				\draw[thick,orange](1.9,3.3) arc(-180:180:.325cm and .12cm);
				\draw[orange](2.05
				,3.3) to[out=-20,in=200](2.42,3.3);
				\draw[orange](2.1,3.29) to[out=20,in=160](2.38,3.29);
				
				\draw[thick](.2,4) to[out=240,in=80](-.2,3.3)  to[out=260,in=90](-1,1) to[out=270,in=180](1,0) to[out=0,in=270](3,1) to[out=90,in=275](2.55,3.3) to[out=95,in=300](2.25,4);
				\draw[thick](2.25,4) to[out=240,in=85](1.9,3.3) to[out=265,in=0](1,2.5) to[out=180,in=270](.490,3.3) to[out=100,in=290](.2,4);
				\draw[thick](.2,1.5) to[out=-50,in=240](1.8,1.5);
				\draw[thick](.3,1.4) to[out=30,in=140](1.7,1.4);
				\filldraw[blue] (.2,3.95) circle (1.5pt);
				\filldraw[blue] (2.25,3.95) circle (1.5pt);

				\node at ((0.2,4.2){$\scriptstyle p_1$ };
				\node at ((2.25,4.2){$\scriptstyle p_2$ };
				
			\end{tikzpicture}
			\caption{Collapsing the cusps $\mathcal{C}(M)$, $S=\{p_1,p_2\}$}
			\label{fig:Cone}
		\end{subfigure}
		\caption{Schematic picture of the 3 compactifications of $M$.}
		\label{fig:Compactifications}
	\end{figure}
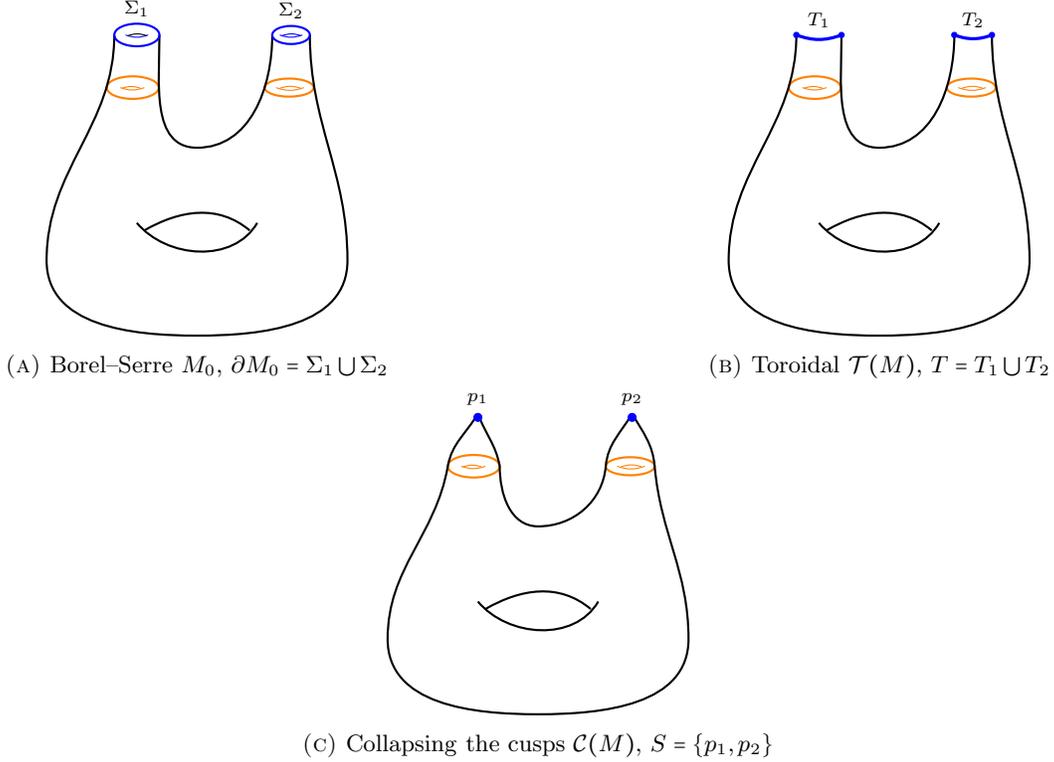

	\begin{lemma}\label{lem:Coneoff}
		The space $\mc{C}(M)$ is the Eilenberg–MacLane space $K(\overline{\Gamma},1)$.
	\end{lemma}
	\begin{proof} Since a neighborhood of each $p_i$ is homeomorphic to the cone on $\Sigma_i$, it follows from van Kampen's theorem that $\pi_1(\mc{C}(M))=\overline{\Gamma}$. It remains to show that $\mc{C}(M)$ is aspherical.
		Consider the short exact sequence \[1\rightarrow K\rightarrow\Gamma\rightarrow\overline{\Gamma}\rightarrow 1.\] Find the covering space of $M_0$ corresponding $K$, $\pi_K\colon M_K\rightarrow M$. Since $M_0$ is a deformation retraction of $M$ with the universal cover $\HC^n$, so $M_K$ can be seen as a $K(K,1)$. Observe that the universal cover of $\mc{C}(M)$, say $\widetilde{\mc{C}(M)}$, is obtained from $M_K$ by collapsing boundary components of $M_K$ to points separately. Since each cusp subgroup is contained in $K$, $\pi_K$ projects each boundary component of $M_K$ onto some boundary component of $M_0$. By Theorem \ref{theorem:Dehn}-(3), $K\cong\underset{g\in \overline{\Gamma}}{\text{\Large$\ast$}}(P_{1}*\cdots*P_{r})$, also recall each $\Sigma_i$ is a $K(P_i,1)$, it follows that for $j\ge0$,
		$$H_j(M_K)\cong H_j(K;\Z)\cong\bigoplus_{g\in \overline{\Gamma}}\left(H_j(P_1;\Z)\oplus\cdots\oplus H_*(P_r;\Z)\right)\cong\bigoplus_{\Sigma\subseteq \partial M_k}H_j(\Sigma).$$
		
		Since $\widetilde{\mc{C}(M)}$ is simply-connected (and connected), by the Hurewicz theorem it is enough to show that $H_i(\widetilde{\mc{C}(M)})=0$ when $i\geq 2$. Let $\overline V$ be a neighborhood of $S$ which is evenly coveorange by the covering map $\pi\colon \widetilde{\mc{C}(M)}\rightarrow \mathcal{C}(M)$. Construct an open cover of $\widetilde{\mc{C}(M)}$, $\{U_i\}_{i=1}^2$, where $U_1:=p^{-1}(\mc{C}(M)\setminus S)$ and $U_2:=p^{-1}(\bar V)$. Note each component of $U_2$ is homeomorphic to the cone on some $\Sigma_i$. We choose $U_1$ and $U_2$ so that $U_1\cap U_2$ is a disjoint union of connected components. Each component is homeomorphic to $\Sigma\times (0,1)$ for some $\Sigma\subseteq \partial M_K$.
		It follows that $H_i(U_1\cap U_2)\cong\bigoplus_{\Sigma\subset \partial M_K}H_p(\Sigma) $.
		
		Now apply Mayer--Vietoris to the triad $(\widetilde{\mc{C}(M)},U_1,U_2)$, it gives a long exact sequence:
		$$... \to H_{i}(U_1\cap U_2)\to H_i(U_1)\oplus H_i(U_2)\to H_i(\widetilde{\mc{C}(M)})\to H_{i-1}(U_1\cap U_2)\to H_{i-1}(U_1)\oplus H_{i-1}(U_2)\to...$$
		Observe that each connected component of $U_2$ as a cone is contractible, thus $H_i(U_2)=0$ for $i>0$. On the other hand, 
		$H_i(U_1)\cong H_1(M_K)$ for all $i$ since $U_1$ is homeomorphic to $M_K$. Also recall that $H_i(U_1\cap U_2)\cong \bigoplus_{\Sigma\subset \partial M_K}H_i(\Sigma)$, so the long exact sequence becomes 
		$$...\to \bigoplus_{\Sigma\subseteq \partial M_K}H_{i}(\Sigma)\to H_i(M_K)\to H_i(\widetilde{\mc{C}(M)})\to \bigoplus_{\Sigma\subseteq \partial M_K}H_{i-1}(\Sigma)\to H_{i-1}(M_K)\to...$$
		for all $i\geq 2$. In the long exact sequence above, since $H_j(M_K)\cong\bigoplus_{\Sigma\subseteq \partial M_k}H_j(\Sigma)$, it follows that $H_i(\widetilde{\mc{C}(M)})=0$ for all $i \geq 2$.
	\end{proof}

	\section{Proof of the Theorem \ref{maintheoremnew}}\label{ending}
	Recall Lemma \ref{lem:Coneoff} says $H_*(\overline{\Gamma};\Z)\cong H_*(\mc{C}(M))$. Observe that $\mc{C}(M)$ is obtained from $M_0$ by collapsing boundary components to points separately. Denote such a quotient map by $q\colon M_0\rightarrow \mc{C}(M)$, which also induces a map of pairs $(M_0,\partial M_0)\to (\mathcal{C}(M),S)$.

	\begin{lemma}\label{lemmaQM} The induced map $q_*\colon H_{2n}(M_0,\partial M_0)\rightarrow H_{2n}(\mathcal{C}(M),S)$ is an isomorphism.  In particular, $H_{2n}(\mc{C}(M))$ is isomorphic to $H_{2n}(M_0,\partial M_0)$ and $H_{2n}(\overline \Gamma)\cong\mathbb Z$.
	\end{lemma}
	\begin{proof} Note $(M_0,\partial M_0)$ and $(\mathcal{C}(M),S)$ are good pairs, it follows from the Theorem 2.13 of \cite{Hatcher}), the quotient maps $ M_0\to M_0/\partial M_0$ 
		and $\mc{C}(M)\to \mc{C}(M)/S$ induce the isomorphism $$H_{2n}(M_0,\partial M_0)\to \tilde H_{2n}(M_0/\partial M_0,\{*\}),$$ where $\{*\}$ is the image of $\partial M_0$, and the isomorphism $$H_{2n}(\mathcal{C}(M),S)\to\tilde H_{2n}(\mathcal{C}(M)/S,\{*\}),$$ for each $i\ge0$, where now $\{*\}$ represents the common image of $\{c_1,\ldots, c_r\}$ respectively. There is an obvious homeomorphism $ (M_0/\partial M_0,*)\to (\mc{C}(M)/S,*)$, which induces an isomorphism
		$$H_{2n}(M_0/\partial M_0,*)\to H_{2n}(\mc{C}(M)/S,*).$$
		Therefore, there is a commutative diagram
		\[
		\xymatrix{H_{2n}(M_0,\partial M_0)\ar[r]^{ q_*}\ar[d]_\cong &H_{2n}(\mc{C}(M),S)\ar[d]^{\cong}\\
			H_{2n}(M_0/\partial M_0,*)\ar[r]_\cong & H_{2n}(\mc{C}(M)/S,*)
		}\]
		It follows that the induced map $q_*:H_{2n}(M_0,\partial M_0)\rightarrow H_{2n}(\mc{C}(M))$ is also an isomorphism. 
		Since the dimension of $\mc{C}(M)$ is $2n\geq 4$, the long exact sequence of the relative homology gives the isomorphism $H_{2n}(\mc{C}(M))\to H_{2n}(\mathcal{C}(M),S)$ induced by inclusion. Therefore, $$H_{2n}(M_0,\partial M_0)\to H_{2n}(\mc{C}(M))$$ is an isomorphism. By Poincar\'e-Lefschetz duality, $H_{2n}(M_0,\partial M_0)\cong \Z$. Since the dimension of $\mc{C}(M)$ is $2n\ge 4$, from the long exact sequence of the pair $(\mc{C}(M),S)$, there is an isomorphism induced by inclusion: $H_{2n}(\mc{C}(M))\xrightarrow{\cong}H_{2n}(\mc{C}(M),S)$. By Lemma \ref{lem:Coneoff}, $\mc{C}(M)$ is a $K(\overline\Gamma,1)$, it follows that $H_{2n}(\overline \Gamma)\cong\mathbb Z$.
		
	\end{proof}

	We are now ready to prove Theorem \ref{maintheoremnew}. 
	
	\begin{proof}[Proof of Theorem \ref{maintheoremnew}]
		Suppose $(\Gamma,\cP)$ admits Property \ref{property} on a CAT(0) cube complex $X$. The strategy is to find a hyperbolic Dehn filling of $\Gamma$ which is cubulated and which factors through $\pi_1(\mathcal{T}(M))$. Define $\mc{D}$ to be the image of $\mc{P}$ under homomorphism on $\pi_1$ induced by $f$, which is a collection of free abelian subgroups of $\pi_1(\mathcal{T}(M))$ such that the pair $(\pi_1(\mathcal{T}(M)),\mc{D})$ is relatively hyperbolic.
		
		Note since $K$ is a normal subgroup containing each peripheral subgroup, and $\mc{T}(M)$ is obtained by collapsing  a subgroup of each peripheral subgroup, we have the desired factorization:
		\begin{equation}\label{eqn:KeyDiagram}
			\xymatrix{(\Gamma,\cP)\ar[rr]^\varphi\ar[dr]_{f_*}&&(\overline{\Gamma},1)\\&(\pi_1(\mathcal{T}(M)),\mathcal{D})\ar[ur]_\kappa}
		\end{equation}
		
		Since $\pi_1(\mathcal{T}(M))$ is K\"ahler and $\Gamma$ is cubulated hyperbolic, in particular, it admits a full system of convex cuts, by the corollary on page 52 of \cite{delzant2005cuts}, the homomorphism $\kappa: \pi_1(\mathcal{T}(M))\to \overline{\Gamma}$ must virtually admit a further factorization
		\begin{equation}\label{eqn:Surface Factorization}
			\xymatrix{(\pi_1(\mathcal{T}(M))_1,\mc{D}_1)\ar[rr]^{\kappa_1}\ar[dr]&&(\overline{\Gamma},1)\\&(\overline{\pi_1(\mathcal{T}(M))},\overline{\mc{D}})\ar[ur]}
		\end{equation}
		where $\pi_1(\mathcal{T}(M))_1$ is a finite index subgroup of $\pi_1(\mathcal{T}(M))$, $\overline{\pi_1(\mathcal{T}(M))}$ is a surface group, and the image $\overline{\mc{D}}$ of $\mc{D}_1$ (where $\mc{D}_1$ is induced from $\mc{D}$) is a collection of quasi-convex subgroups of $\overline{\pi_1(\mathcal{T}(M))}$.

		Since the virtual cohomological dimension of $\overline{\pi_1(\mathcal{T}(M))}$ is 2, but $\dim(\mc{T}(M))=2n\geq 4$, this implies that $(\kappa_1)_*\colon H_{2n}(\pi_1(\mathcal{T}(M))_1,\mc{D}_1)\rightarrow H_{2n}(\overline{\Gamma})$ must be  the zero map. On the other hand, recall $M_0$ is the $K(\Gamma,1)$,$\mathcal{T}(M)$ is the $K(\pi_1(\mathcal{T}(M)),1)$, and $\mathcal C(M)$ is the $K(\overline\Gamma,1)$, so by Lemma \ref{lemmaQM}, $\varphi$ induces an isomorphism $\varphi_*\colon H_{2n}(\Gamma,\cP)\to H_{2n}(\overline\Gamma,\{1\})\cong H_{2n}(\overline\Gamma)\cong\Z$, which factors through the isomorphism  $f_*\colon H_{2n}(\Gamma,\cP)\rightarrow H_{2n}(\pi_1(\mathcal{T}(M)),\mc{D})$ by Proposition \ref{HomologyToroidal}. 
		
		Therefore $\kappa_*\colon H_{2n}(\pi_1(\mathcal{T}(M)),\mc{D})\rightarrow H_{2n}(\overline{\Gamma},\{1\})$ has to be an isomorphism. In particular, the restriction map $(\kappa_1)_*=({\kappa|_{\pi_1(\mathcal{T}(M))_1}})_*\colon H_{2n}(\pi_1(\mathcal{T}(M))_1,\mc{D}_1)\rightarrow H_{2n}(\overline{\Gamma},\{1\})\cong\Z$ induced from the degree map is non-trivial, which is a contradiction.
	\end{proof}

	
	

	\bibliography{name} 
	\bibliographystyle{alpha} 
	
\end{document}